\documentclass{amsart}
\usepackage{amssymb}

\newtheorem{theorem}{Theorem}[section]
\theoremstyle{plain}

\newtheorem{construction}{Construction}

\newtheorem{lemma}[theorem]{Lemma}

\numberwithin{equation}{section}

\newcommand{\SUT}{\mbox{ such that }}
\newcommand{\OR}{\mbox{ or }}
\newcommand{\AND}{\mbox{ and }}

\newcommand{\IFF}{\mbox{ if and only if }}

\newcommand{\nin}{\not\in}

\newcommand{\ST}{\mathcal{T}}
\newcommand{\SV}{\mathcal{V}}
\newcommand{\SE}{\mathcal{E}}
\newcommand{\SA}{\mathcal{A}}

\newcommand{\SM}{\mathcal{M}}
\newcommand{\SN}{\mathcal{N}}
\newcommand{\SG}{\mathcal{G}}

\newcommand{\restrict}{\upharpoonright}

\usepackage{graphicx}
\newcommand{\showpic}[4]{
\begin{figure}[h]
\begin{center}
\includegraphics[width=#1in]{#2.pdf}
\caption{#3}
\label{#4}
\end{center}
\end{figure}
}

\newcommand{\N}{\mathbb{N}}

\newcommand{\ec}[1]{\overline{#1}}

\newcommand{\dom}[2]{{#1} \rightarrow {#2}}
\newcommand{\ndom}[2]{{#1} \not \rightarrow {#2}}

\newcommand{\ecsig}[2]{\Sigma(#1,#2)}
\newcommand{\sig}[2]{\ecsig{#1}{#2}}

\newcommand{\tc}[1]{\widehat{#1}} 
\newcommand{\tr}[1]{\Upsilon({#1})} 

\newcommand{\brt}[1]{\mathbf{1}_{#1}} 
\newcommand{\rt}[1]{1_{#1}} 

\newcommand{\pn}[2]{[#1]_{#2}}

\newcommand{\pnbu}[1]{\pn{\brt{U}}{#1}}
\newcommand{\pnbv}[1]{\pn{\brt{V}}{#1}}

\newcommand{\startthing}[1]{\begin{#1}}
\newcommand{\stopthing}[1]{\end{#1}}

\newcommand{\bprf}{\startthing{proof}}
\newcommand{\eprf}{\stopthing{proof}}
\newcommand{\blem}{\startthing{lemma}}
\newcommand{\elem}{\stopthing{lemma}}

\newcommand{\bitem}{\begin{itemize}}
\newcommand{\eitem}{\end{itemize}}
\newcommand{\benum}{\begin{enumerate}}
\newcommand{\eenum}{\end{enumerate}}
\newcommand{\bques}{\begin{question}}
\newcommand{\eques}{\end{question}}

\newcommand{\define}[1]{\emph{#1}}

\title{Gallai Multigraphs}
\author[K. Pula]{Kyle Pula}
\email[Pula]{jpula@math.du.edu}

\address{Department of Mathematics \\
University of Denver \\
2360 S Gaylord St \\
Denver, CO 80208, U.S.A.}

\begin{document}
\maketitle

\begin{abstract}
A complete edge-colored graph or multigraph is called Gallai if it lacks rainbow triangles. We give a construction of all finite Gallai multigraphs.
\end{abstract}

\section{Background}

A complete, edge-colored graph without loops lacking rainbow triangles is called \define{Gallai} after Tibor Gallai, who gave an iterative construction of all finite graphs of this sort \cite{gallai}. Some work progress has been made on the more general problem of understanding edge-colored graphs lacking rainbow $n$-cycles for a fixed $n$. In particular, Ball, Pultr, and Vojt\u{e}chovsk\'y give algebraic results about the sequence $(n : G \mbox{ lacks rainbow $n$-cycles} )$ as a monoid \cite{ball}, and Vojt\u{e}chovsk\'y extends the work of Alexeev \cite{alexeev} to find the densest arithmetic progression contained in this sequence \cite{petr}. 

While searching for a general construction of graphs lacking rainbow $n$-cycles, we were confronted with the task of understanding a different generalization. We call a complete, simple, edge-colored multigraph \define{Gallai} if it lacks rainbow triangles. (By complete here we mean only that each pair of vertices is connected by at least one edge.) Mubayi and Diwan make a conjecture about the possible color densities in Gallai multigraphs having at most three colors \cite{Mubayi}. 

The main result of this paper is a construction of all finite Gallai multigraphs.

\subsection{Basic Notation}

We denote vertices using lowercase letters such as $u,v, \AND w$, sets of vertices using uppercase letters such as $U,V, \AND W$, and colors using uppercase letters such as $A,B, \AND C$.  Given two sets of vertices, $U$ and $V$, we write $UV$ for the set of edges connecting vertices of $U$ to vertices of $V$.  This notation will also be used with singletons, $u$ and $v$, to refer to the edges connecting $u$ and $v$.  To denote the set of colors present in a set of edges, say $UV$, we write $\ec{UV}$ when there is no risk of ambiguity. Otherwise we refer explicitly to the coloring at hand, i.e. $\rho[UV]$. If $\ec{UV} = \{A\}$, we will often shorten notation by writing $\ec{UV} = A$. If the edges of an $n$-cycle contain no repeated colors, we say it is \define{rainbow}.

Many of our results will be stated in terms of mixed graphs. A \define{mixed graph} is a triple $M = (V,E,A)$ with vertices $V$, undirected edges $E$, and directed edges $A$. We say $M$ is \define{complete} if every pair of distinct vertices is connected by a single directed or undirected edge. The \define{weak components} of a directed graph are the components of the graph that results from replacing each directed edge with an undirected edge. For our purposes, the \define{weak components} of a mixed graph $M = (V,E,A)$ will be the weak components of the directed graph $(V,A)$. Note that this notion of component disregards undirected edges.

We use the term \define{rooted tree} to refer to a directed graph that is transitive and whose transitive reduction forms a tree in the usual sense. If $(V, A)$ is a rooted tree, then its \define{root}, written $\rt{V}$, is the unique vertex having the property that there is a directed edge from $\rt{V}$ to every other vertex in $V$.

\subsection{Construction of Gallai Graphs}

It is easy to see that the following construction yields Gallai graphs.

Let $(G = (V, E), \rho)$ be a complete, edge-colored graph such that $|\rho(VV)| \leq 2$. For every $v \in V$, let $(G_v = (V_v, E_v), \rho_v)$ be a Gallai graph. Construct a new complete graph on $\bigcup_{v \in V}V_v$ with edge-coloring $\rho'$ defined by
\begin{align*}
\rho'(xy) := 
\left\{\begin{array}{lrl}
\rho_v(xy) & \mbox{if } & x, y \in V_v \\
\rho(vw) & \mbox{if } & x \in V_v, y \in V_w, \AND v \neq w \end{array}\right. .
\end{align*}

Gallai showed \cite{gallai} that every finite Gallai graph can be built iteratively by the above construction, i.e. for every Gallai graph, $G = (V,E)$,  there exists a nontrivial partition $V = \cup V_i$ such that $|\ec{V_i V_j}| = 1$ for $i \neq j$ and $|\cup_{i \neq j} \ec{V_i V_j}| \leq 2$. 

\section{Decomposition of Gallai Multigraphs}\label{SEC:DECOMP}

While our purpose is to give a construction of Gallai multigraphs, the most substantive step is in developing the appropriate decomposition. Before stating this result, we describe our basic techniques and introduce a few definitions.

\subsection{Basic Techniques: Maximality and Dominance}

Let $(G = (V,E), \rho)$ be a Gallai multigraph.  We will in all cases assume that distinct edges connecting the same vertices are colored distinctly. We also think of $V \subseteq \N$ and thus having a natural ordering. We say that $(G = (V,E), \rho)$ is \define{uniformly colored} if $\rho(e_1) = \rho(e_2)$ for all $e_i \in E$.

We call $uv$ \define{isolated} if for every $w \nin \{u,v\}$, $\ec{uw} = \ec{vw}$ and $|\ec{uw}| = 1$. Notice that if $uv$ is isolated we can reduce the multigraph by collapsing the edge(s) $uv$. Likewise, given any multigraph, we can arbitrarily introduce new isolated edges without introducing rainbow triangles. We therefore call a multigraph \define{reduced} if it contains no isolated edges.

We call $uv$ \define{maximal} if no new color can be added to $\ec{uv}$ without introducing a rainbow triangle. Here we allow the possibility that $uv$ has ``all possible colors'' and thus is maximal. Likewise, $(G = (V,E), \rho)$ is \define{maximal} if $uv$ is maximal for all $u,v \in V$.

Let $(G = (V,E), \rho)$ be a maximal Gallai multigraph. For $u,v \in V$, notice that $|\ec{uv}| \geq 3 \IFF uv$ is isolated. Therefore, if $G$ is reduced, $|\ec{uv}| = 1 \OR 2$ for all $u,v \in V$. Furthermore, if $G$ is not reduced, we can reach a reduced Gallai multigraph by successively collapsing isolated edges of $G$.

To construct all Gallai multigraphs, it therefore suffices to construct the reduced maximal ones. In Section \ref{SEC:DECOMP}, we develop a basic decomposition of any maximal reduced Gallai multigraph and then in Section \ref{SEC:CONSTRUCTION} reverse this decomposition to construct all finite reduced Gallai multigraphs.

\blem\label{L:MAX}
Suppose $(G = (V,E), \rho)$ is a maximal Gallai multigraph.  If $u,v \in V$ and $A \in \ec{uv}$, then for all $B \nin \ec{UV}$, there is $w \in V \setminus \{u,v\}$ and $C \nin \{A,B\}$ such that either $A \in \ec{uw} \AND C \in \ec{wv} \OR C \in \ec{uw} \AND A \in \ec{wv}$.
\elem

\bprf
Since $G$ is maximal and $B \nin \ec{uv}$, we can find $w \neq u,v$ such that $u,v,w$ would form a rainbow triangle if $B$ were to be added to $\ec{uv}$.  Thus we may find $X \in \ec{uw} \AND Y \in \ec{vw}$ such that $X, Y, \AND B$ are distinct.  However, since $A \in \ec{uv}$, $|\{X, Y, A\}| \leq 2$ and thus $A = X$ or $A = Y$.  Let $C$ be the other color.
\eprf

While Theorem \ref{T:BASECASE} will follow from Theorem \ref{T:DECOMP}, our general decomposition result, we present it separately here because of its importance in understanding the most basic structure of a maximal reduced Gallai multigraph.

\begin{theorem}\label{T:BASECASE}
The vertices of a reduced maximal Gallai multigraph that are connected by two edges form uniformly colored cliques.
\end{theorem}
\begin{proof}
Let $(G = (V,E), \rho)$ be a reduced maximal Gallai multigraph. Let $u,v,w \in V$. Suppose $\ec{uv} = \{A,B\} \AND \ec{vw} = \{C,D\}$. If $\{A,B\} \neq \{C,D\}$, then we find a rainbow triangle no matter the colors of $\ec{uw}$. Suppose then that $\ec{uv} = \ec{vw} = \{A,B\}$. Certainly $\ec{uw} \subseteq \{A,B\}$. Suppose $\ec{uw} = A$. Then by Lemma \ref{L:MAX}, we may find $x \in V \setminus \{u,w\}$ such that, without loss of generality, $A \in \ec{ux} \AND C \in \ec{wx}$. Then $v,x,w$ contains a rainbow triangle.
\end{proof}

Theorem \ref{T:DECOMP} is primarily an explanation of how each of these uniformly colored cliques are related to each other, and the following relation on sets of vertices plays a central role in this analysis. Let $(G = (\SV,\SE), \rho)$ be a Gallai multigraph.  For $U, V \subseteq \SV$ disjoint, we say that $U$ \define{dominates} $V$ and write $\dom{U}{V}$ iff $|\ec{UV}| > 1$ and
\benum
\item $U = \{u\}, V = \{v\}$ and $u < v$ or 
\item $|U| > 1 \OR |V| > 1$ and for every $u \in U \AND v \in V$, $\ec{uv} = \ec{uV}$.
\eenum
Given $U, V \subseteq \SV$, we write $\Sigma(U,V)$ for the map from $U$ to the powerset of $\ec{UV}$ defined by $u \mapsto \ec{uV}$. When $\dom{U}{V}$, $\Sigma(U,V)$ completely describes the relationship between $U$ and $V$ and we call it the \define{signature of $\dom{U}{V}$}.

Given a reduced maximal Gallai multigraph $(G = (\SV,\SE), \rho)$, we will describe its structure through a sequence of edge-colored mixed graphs $M_n(G) = (\SV_n, \SE_n, \SA_n)$ defined as follows:
\benum
\item $\SV_0 := \SV$,
\item $\SA_0 := \{(u,v) \in \SV^2 : \dom{u}{v} \}$, and
\item $\SE_0 := \{\{u,v\} \in [\SV]^2 : |\rho[uv]| = 1 \}$,
\eenum
and for $n \geq 1$
\benum
\item[($1^\prime$)] $\SV_n$ is the partition of $\SV$ induced by the weak components of $M_{n-1}(G)$,
\item[($2^\prime$)] $\SA_n := \{(U,V) \in \SV_n ^2 : \dom{U}{V} \}$, and
\item[($3^\prime$)] $\SE_n := \{\{U,V\} \in [\SV_n]^2 : |\rho[UV]| = 1 \}$.
\eenum

For each $n$, $\rho$ induces a list-coloring, $\rho'$, of $\SE_n \cup \SA_n$ by $\rho'(e) = \rho[UV]$ where $e = (U,V)$ or $e = \{U,V\}$. Likewise, $\Sigma$ induces a partition of $\SA_n$ by $(U_1, V_1) \sim_\Sigma (U_2, V_2) \IFF \Sigma(U_1, V_1) = \Sigma(U_2, V_2)$. Note that $(U_1, V_1) \sim_\Sigma (U_2, V_2) \IFF U_1 = U_2 \AND \rho[uV_1] = \rho[uV_2]$ for all $u \in U_1$.

Figure (1) shows an example of this sequence for a particular Gallai multigraph. For readability, we show only those edges in $M_n(G)$ that contribute to the formation of directed edges in $M_{n+1}(G)$. The hash marks on the directed edges in $M_1(G)$ indicate whether the signatures agree or disagree.

\showpic{4.7}{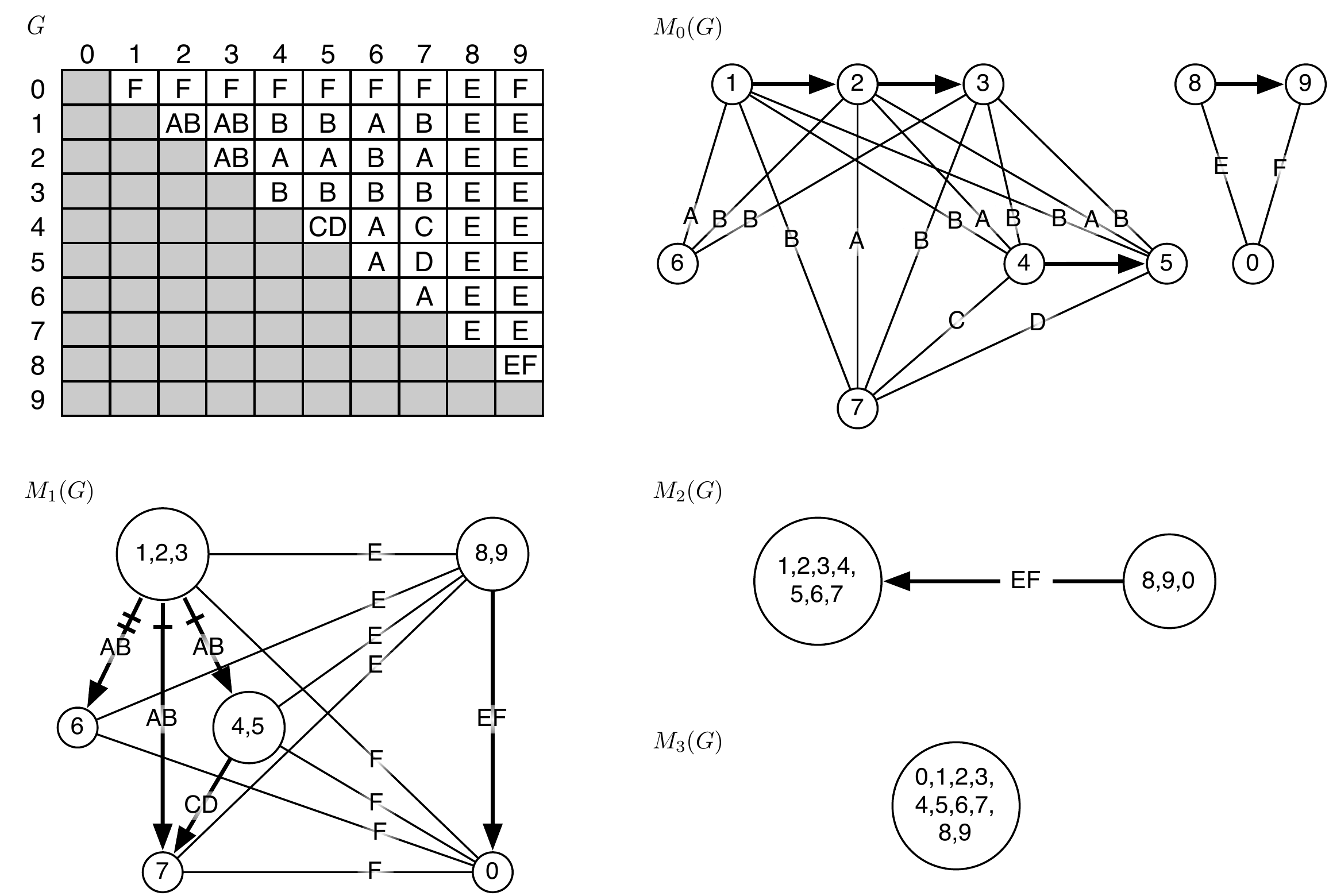}{Sequence of $M_n(G)$ for a Gallai multigraph.}{Example_Decomp}

\subsection{Decomposition of Maximal Gallai Multigraphs}\label{SEC:DECOMP}

We may now state our main result.
\begin{theorem}\label{T:DECOMP}
Let $G$ be a reduced maximal Gallai multigraph, $H$ an induced subgraph of $G$, and $M_n(H) = (\SV_n, \SE_n, \SA_n)$ the sequence described above. Then
\benum
\item $M_n(H)$ is complete,
\item
$|\rho'(e)| = \left\{\begin{array}{lll}
1 & \mbox{if } & e \in \SE_n \\
2 & \mbox{if } & e \in \SA_n
\end{array}\right.$,
\item the weak components of $M_n(H)$ are rooted trees, and
\item if $(U,V), (V,W) \in \SA_n,$ then $(U,V) \sim_\Sigma (U,W)$
\eenum
for all $n \geq 0$.
\end{theorem}

For convenience, if $M_k(H)$ has properties (1)-(4) for all $k \leq n$, we will say that $H$ has the \define{tree property} for $n$.

As we are primarily interested in decomposing reduced maximal Gallai multigraphs, our most important application of Theorem \ref{T:DECOMP} is when $H = G$. We will, however, need the result in this greater generality in a key technical step in Section \ref{MAKING_TREES}.  

\subsection{Proof of Theorem \ref{T:DECOMP}}

Throughout this section, we assume $(G = (\SV,\SE), \rho)$ is a reduced maximal Gallai multigraph and $H$ is an induced subgraph of $G$.

\blem \label{L:ABC}
Suppose $U, V, W \subseteq \SV$ disjoint, $\dom{U}{V}$, and $\{A,B\} = \ec{UV}$.
\begin{enumerate}
\item If $\ec{UW} = C \notin \{A,B\}$, then $\ec{V W} = C$.
\item If $\ec{V W} = C \notin \{A,B\}$, then either $C \in \ec{U W}$ or $\dom{U}{W}$ and $\Sigma(U, V) = \Sigma(U, W)$.  If we also know that either $\dom{U}{W}, \dom{W}{U}, \OR |\ec{UW}| = 1$ and that $U$ always dominates with the same colors (i.e., whenever $\dom{U}{U^\prime}$, then $\ec{U U^\prime} = \{A,B\}$), then either $\ec{UW} = C$ or $\dom{U}{W}$ and $\Sigma(U, V) = \Sigma(U, W)$.

\item[$(2^\prime)$] If $W$ is a single vertex, we need only require $C \in \ec{V W}$ in $(2)$.
\end{enumerate}
\elem
\bprf
$(1)$ Fix $v \in V$ and $w \in W$.  Since $\dom{U}{V}$, we may select $u_A, u_B \in U$ such that $A \in \ec{u_A v}$ and $B \in \ec{u_B v}$.  Observe that the triangle $w, u_A, v$ forces $\ec{v w} \subseteq \{A, C\}$ while $w, u_B, v$ forces $\ec{v w} \subseteq \{B, C\}$.  Thus $\ec{v w} = C$.  Since $v$ and $w$ were arbitrary, $\ec{V W} = C$.

$(2)$ Fix $u \in U, w \in W, v \in V$. Since $\ec{UV} = \{A,B\}$. We are in one of the following cases: $\ec{uv}=A$, $\ec{uv}=B$, or $\ec{uv}=\{A,B\}$. If $\ec{uv} = \{A,B\}$, then $\ec{vw} = C$ forces $\ec{uw} = C$. Suppose then that $C \nin \ec{UW}$. If $\ec{uv} = A$, then $\ec{uw} \subseteq \{A,C\}$ and thus $\ec{uw} = A$. Likewise, if $\ec{uv} = B$, then $\ec{uw} = B$. We thus have either $C \in \ec{UW} \OR \dom{U}{W}$ and $\Sigma(U,V) = \Sigma(U,W)$.

Suppose we also know that we are in one of the following cases:
\benum
\item[(i)] $\dom{U}{W} \AND \ec{UW} = \{A,B\}$,
\item[(ii)] $\dom{W}{U}$, or
\item[(iii)] $|\ec{UW}| = 1$.
\eenum

Again, if $C \nin \ec{UW}$, then we must be in case (i). We would like to conclude that if $C \in \ec{UW}$, then we are in case (iii) and thus $\ec{UW} = C$. Suppose $\dom{W}{U}$. To avoid a rainbow triangle, $\ec{UW} \subseteq \{A,B,C\}$. In particular, since $|\ec{UW}| \geq 2$, $A \in \ec{UW} \OR B \in \ec{UW}$. Assume $A \in \ec{UW}$ and fix $w \in W$ such that $A \in \ec{wU}$. We may then choose $u \in U \AND v \in V$ such that $B \in \ec{uv}$ but now $u,v,w$ is a rainbow triangle. Thus $\ndom{W}{U}$ and $\ec{UW} = C$.

$(2^\prime)$ Let $W = \{w\}$ and $V = \{v\}$ and repeat the proof of (2).
\eprf

\blem\label{L:BASECASE}
$H$ has the tree property for $0$.
\elem
\bprf
It is clear that $M_0(H)$ is complete. The rest of the claim is essentially a restatement of Theorem \ref{T:BASECASE}. By the definition of dominance between single vertices, each complete, uniformly colored clique from Theorem \ref{T:BASECASE} becomes a linear ordered set of vertices and thus a rooted tree. In this context, property (4) of Theorem \ref{T:DECOMP} is simply the observation that these cliques are uniformly colored.
\eprf

Before proceeding, we introduce some convenient notation. Elements of $\SV_n$ are by definition subsets of $V(H)$.  We will however at times want to speak of their structure as rooted trees. For $U \in \SV_n$, we write $\tr{U}$ to refer to the set of elements of $\SV_{n-1}$ contained in $U$ and $\rt{U}$ to refer to the root of $\tr{U}$.  Notice that $\rt{U} \in \SV_{n-1}$ has its own tree structure and thus we may refer to $\rt{\rt{U}}$, $\rt{\rt{\rt{U}}}$, etc. We may continue this recursion until we reach a single vertex.  We write $\brt{U}$ to refer to this single vertex. Similarly, for $u \in V(H)$, we write $[u]_n$ to refer to the unique $U \in \SV_n$ containing $u$. Lastly, we point out how this notation fits together. For $U \in \SV_n$, $\pn{\brt{U}}{n} = U$, $\pn{\brt{U}}{n-1} = \rt{U}$, $\pn{\brt{U}}{n-2} = \rt{\rt{U}}$, ..., and $\pn{\brt{U}}{0} = \brt{U}$.

We also associate a set of colors with each member of $\SV_n$ as follows. For $u \in \SV_0$, $\tc{u} := \cup_{\dom{u}{v}} \ec{uv}$ and for $U \in \SV_{n+1}$, $\tc{U} := \tc{\rt{U}}$ for $n \geq 0$. Lemma \ref{L:DOM_COLORS} demonstrates the importance of this notation.

\blem\label{L:DOM_COLORS}
Suppose $H$ has the tree property for $n$ and $(U, V) \in \SA_{n+1}$. Then $U$ always dominates with the same two colors $\tc{U}$, i.e.
\begin{enumerate}
\item $\ec{UV} = \tc{U}$ and
\item $|\tc{U}| = 2$.
\end{enumerate}
\elem

\bprf

It is clear that $|\tc{U}| = 2$ since $\tc{U}$ is defined inductively and dominance between two vertices must be with exactly two colors. Likewise, (1) certainly holds for $U,V \in \SV_{0}$.

Since $H$ has the tree property for $n$, $|\ec{U' V}| = 1$ for all $U' \in \tr{U}$, and since $\dom{U}{V}$, we know $|\ec{UV}| \geq 2$. Suppose we find $C \in \ec{UV} \setminus \tc{U}$. Then fix $U_C \in \tr{U}$ such that $\ec{U_C V} = C$. If $U_C = \rt{U}$, then for all $U' \in \tr{U} \setminus \{\rt{U}\}$, $\dom{\rt{U}}{U'}$ and we may apply part (1) of Lemma \ref{L:ABC} to get that $\ec{U' V} = C$ and thus $\ec{UV} = C$, a contradiction.

Suppose then that $\dom{\rt{U}}{U_C}$. By induction, $\ec{\rt{U} U_C} = \tc{U}$. We may now apply part (2) of Lemma \ref{L:ABC} to get that either $C \in \ec{\rt{U} \rt{V}}$ or $\dom{\rt{U}}{\rt{V}}$. The former has already been ruled out while the latter contradicts the assumption that $\rt{U}$ and $\rt{V}$ were in different components of $\SV_n$.
\eprf

The following lemma is useful because it allows us to locate a vertex in $U$ that is connected to the rest of $U$ by only the colors contained in $\tc{U}$. 

\blem\label{L:BASE_ROOT_LEMMA}
If $H$ has the tree property for $n$ and $U \in \SV_{n+1}$, then $\ec{U \brt{U}} = \tc{U}$.
\elem
\bprf
By Lemma \ref{L:DOM_COLORS}, $|\tc{U}| = 0 \OR 2$.  If $\tc{U} = \emptyset$, then $U$ is a single vertex, i.e. $U = \{\brt{U}\}$, and thus $\ec{U \brt{U}} = \emptyset$.

For $n=0$, $U$ is either a single vertex, in which case $\tc{U} = \emptyset$, or $U$ is a nontrivial uniformly colored clique, in which case $\tc{U}$ is by definition $\ec{\brt{U}U}$.

For $n \geq 1$, if $U$ is a single vertex, again $\tc{U} = \emptyset$. Otherwise, by induction $\ec{\rt{U} \brt{\rt{U}}} = \tc{\rt{U}} = \tc{U}$. But since $\brt{\rt{U}} = \brt{U}$, we have $\ec{\rt{U} \brt{U}} = \tc{U}$, and by Lemma \ref{L:DOM_COLORS}, $\ec{\rt{U} U'} = \tc{U}$ for every $U' \in \tr{U} \setminus \{\rt{U}\}$. Finally, given that $\brt{U} \in \rt{U}$, $\ec{\brt{U} U'} \subseteq \ec{\rt{U} U'}$ and thus $\ec{\brt{U} U} = \tc{U}$.
\eprf

Lemmas \ref{L:3RD_COLOR_PROPAGATES} and \ref{L:3RD_COLOR_WOUTSIDE} will be used in situations where a tree is connected to another tree or vertex by a color not present in the dominating colors of the first tree.

\blem\label{L:3RD_COLOR_PROPAGATES}
Suppose $H$ has the tree property for $n$, $U,V \in \SV_{n+1}$ distinct, and $V' \in \tr{V}$ such that $C \in \ec{UV'} \setminus \tc{U}$.  Then $\ec{U V'} = C$.
\elem
\bprf
If $\ec{\rt{U} V'} = C$, then for every $U' \in \tr{U} \setminus \{\rt{U}\}$ we may apply part (1) of Lemma \ref{L:ABC} with $\dom{\rt{U}}{U'}$ and $V'$ to get that $\ec{U' V'} = C$ and thus $\ec{UV'} = C$.  Suppose then that $U' \in \tr{U} \setminus \{\rt{U}\}$ such that $\ec{U' V'} = C$.  Applying part (2) of Lemma \ref{L:ABC} with $\dom{\rt{U}}{U'}$ and $V'$, we get that $\ec{\rt{U}V'} = C$ or $\dom{\rt{U}}{V'}$.  Since $\rt{U} \AND V'$ are in different components of $\SV_{n+1}$, we must be in the former case, and by the previous case we are done.
\eprf

\blem\label{L:3RD_COLOR_WOUTSIDE}
Suppose $H$ has the tree property for $n$, $U \in \SV_{n+1}$, and $v \in \SV$ such that $\ec{\brt{U}v} = C \nin \tc{U}$.  Then $\ec{Uv} = C$.
\elem
\bprf
First observe that $v \nin U$ since by Lemma \ref{L:BASE_ROOT_LEMMA}, $\ec{\brt{U}U} = \tc{U}$.  Next let $k$ be maximal such that $\ec{\pn{\brt{U}}{k} v} = C.$  If $k = n+1$, we are done.  Suppose $k < n+1$ and select $u \in \pn{\brt{U}}{k+1} \setminus \pn{\brt{U}}{k}$.  We may apply (1) of Lemma \ref{L:ABC} with $\dom{\pn{\brt{U}}{k}}{u}$ and $v$ to get that $\ec{uv} = C$ and thus violating the maximality of $k$.
\eprf
Note that in Lemma \ref{L:3RD_COLOR_WOUTSIDE} we do not require that $v$ be in $V(H)$ but rather in the larger set $\SV$.

\blem\label{L:DOM_LEMMA}
For $n \geq 0$, suppose $H$ has the tree property for $n$ and $U,V \in \SV_{n+1}$ such that $\ec{UV} \subseteq \tc{U} = \tc{V} = \{A,B\}$  Then the following statements are equivalent:
\begin{enumerate}
\item $\dom{U}{V}$,
\item there is $x \in \SV \setminus {(U \cup V)}$ such that $\dom{U}{x}$ and $\ec{Vx} = C \nin \{A,B\}$, and
\item $\ndom{V}{U}$ and $\ec{UV} = \{A,B\}$.
\end{enumerate}
If the statements are true, then $\sig{U}{V} = \sig{U}{x}$. 
\elem
\bprf
$(1 \Rightarrow 2)$
We may assume $\ec{\rt{U}\rt{V}} = A$.  Since $G$ is maximal and $B \nin \ec{\brt{U}\brt{V}}$, there is $x \in \SV$ such that $A \in \ec{\brt{U}x}$ and $C \in \ec{x\brt{V}}$ or $C \in \ec{\brt{U}x}$ and $A \in \ec{x\brt{V}}$ for $C \nin \{A,B\}$.  Notice that if $C \in \ec{\brt{U} x}$, since $C \nin \tc{U}$, $\brt{U} x$ cannot contain multiple edges and thus $C = \ec{\brt{U} x}$. Likewise, if $C \in \ec{\brt{V} x}$. Furthermore, in either case, since $C \nin \ec{\brt{U}U} = \ec{\brt{V}V}$, we must conclude that $x \in \SV \setminus (U \cup V)$.  Suppose we are in the latter case, i.e. $C = \ec{\brt{U} x}$ and $A \in \ec{x \brt{V}}$.  By Lemma \ref{L:3RD_COLOR_WOUTSIDE}, $\ec{Ux} = C$.  We may then apply part (1) of Lemma \ref{L:ABC} with $\dom{U}{V}$ and $x$ to get that $\ec{Vx} = C$, which contradicts our assumption that $A \in \ec{\brt{V} x}$.

We must then be in the former case, i.e. $A \in \ec{\brt{U}x}$ and $C = \ec{x\brt{V}}$ and, again by Lemma \ref{L:3RD_COLOR_WOUTSIDE}, $\ec{Vx} = C$. Note that if we can show that $C \nin \ec{Ux}$, we may then apply part (2) of Lemma \ref{L:ABC} with $\dom{U}{V}$ and $x$ to get that $\dom{U}{x}$ and that $\sig{U}{V} = \sig{U}{x}$, which is exactly what we would like to prove.

To this end, suppose $C \in \ec{Ux}$ and let $k$ be minimal such that $C \in \ec{\pnbu{k} x}$.  If $k=0$, we have that $C \in \ec{\brt{U} x}$ and it again follows that $\ec{Ux} = C$, which is a contradiction. Thus $k > 0$.  Fix $u \in \pnbu{k}$ such that $C \in \ec{ux}$.  Since $k$ is minimal, $u \in \pnbu{k} \setminus \pnbu{k-1}$ and thus $\dom{\pnbu{k-1}}{u}$. Recall that $\tc{\pnbu{i}} = \tc{U}$ for all $i \leq n$ and thus $\ec{\pnbu{k-1} u} = \{A,B\}$.  

We may now apply part ($2^\prime$) of Lemma \ref{L:ABC} with $\dom{\pnbu{k-1}}{u}$ and $x$ to get that either $C \in \ec{\pnbu{k-1} x}$ or $\dom{\pnbu{k-1}}{x}$ and $\sig{\pnbu{k-1}}{x} = \sig{\pnbu{k-1}}{u}$.  By the minimality of $k$, we must be in the latter case.  Then we may apply part (2) of Lemma \ref{L:ABC} with $\dom{\pnbu{k-1}}{x}$ and $V$ to get that either $C \in \ec{\pnbu{k-1} V}$ or $\dom{\pnbu{k-1}}{V}$. Both of theses cases contradict the assumption that $\ec{\rt{U} V} = A$.  Thus $C \nin \ec{Ux}$.

$(2 \Rightarrow 3)$
We may apply part (2) of Lemma \ref{L:ABC} with $\dom{U}{x}$ and $V$ to get that either $C \in \ec{UV}$ or $\dom{U}{V}$ and $\sig{U}{x} = \sig{U}{V}$.  Since $C \nin \ec{UV} \subseteq \{A,B\}$, we are left with the latter case; $\dom{U}{V}$ and thus $\ec{UV} = \tc{U} = \{A,B\}$ and $\ndom{V}{U}$.

$(3 \Rightarrow 1)$.
Again we may assume $\ec{\rt{U} \rt{V}} = A$.  As argued in $(1 \Rightarrow 2)$, we may find $x \in \SV \setminus (U \cup V)$ such that either $A \in \ec{\brt{U}x}$ and $\ec{Vx} = C$ or $\ec{Ux} = C$ and $A \in \ec{\brt{V}x}$ for some $C \nin \{A,B\}$.  Suppose we are in the latter case.  Since $\ndom{V}{U}$ and $\ec{UV} = \{A,B\}$, there must be $U_A, U_B \in \tr{U}$ and $V' \in \tr{V}$ such that $\ec{U_A V'} = A$ and $\ec{U_B V'} = B$.  Then $x, U_A, V'$ forces $\ec{x V'} \subseteq \{A,C\}$ while $x U_B V'$ forces $\ec{x V'} \subseteq \{B,C\}$.  Thus $\ec{x V'} = C$ and $C \in \ec{Vx}$.

We may then let $k$ be minimal such that $C \in \ec{\pnbv{k} x}$.  If $k = 0$, then $C \in \ec{\brt{V} x}$ and $\ec{Vx} = C$, which contradicts our assumption that $A \in \ec{\brt{V}x}$.  Therefore $k > 0$.  As before, we select $v \in \pnbv{k} \setminus \pnbv{k-1}$ such that $C \in \ec{vx}$.  We now apply part ($2^\prime$) of Lemma \ref{L:ABC} with $\dom{\pnbv{k-1}}{v}$ and $x$ to get that either $C \in \ec{\pnbv{k-1} x}$ or $\dom{\pnbv{k-1}}{x}$.  By the minimality of $k$, we must be in the latter case and we may apply part (2) of Lemma \ref{L:ABC} with $\dom{\pnbv{k-1}}{x}$ and $\rt{U}$ (recall our assumption that $\ec{Ux} = C$) to get that either $C \in \ec{\rt{U} \pnbv{k-1}}$ or $\dom{\pnbv{k-1}}{\rt{U}}$.  Both of these cases contradict the assumption that $\ec{\rt{U} \rt{V}} = A$.

We therefore may assume that $A \in \ec{\brt{U}x}$ and $\ec{Vx} = C$.  It is either the case that $\dom{U}{V}$ or $\ndom{U}{V}$.  If we suppose that $\ndom{U}{V}$, then we are in the case just handled with the roles of $U$ and $V$ reversed.  Since that assumption leads to a contradiction, we have that $\dom{U}{V}$.
\eprf

\blem\label{L:PN_IS_DOM_REGULAR}
If $H$ has the tree property for $n$, then $M_{n+1}(H)$ is complete.
\elem

\bprf
Let $U,V \in \SV_{n+1}$. If $|\ec{UV}| = 1$, then $\{U,V\} \in \SE_{n+1}$. Suppose then that $|\ec{UV}| > 1$.  Notice that if $\tc{U} = \emptyset$, then $U$ is a single vertex and thus $\dom{V}{U}$ and $(V,U) \in \SA_{n+1}$.

We may therefore assume $|\tc{U}| = |\tc{V}| = 2$ and consider the following cases:

Case 1: $\tc{U} \neq \tc{V} \AND |\ec{UV}| > 2$.  We may then select $C_U, C_V \in \ec{UV} $ distinct $\SUT C_U \nin \tc{U} \AND C_V \nin \tc{V}$ and $U' \in \tr{U}, V' \in \tr{V}$ such that $C_V \in \ec{U' V}$ and $C_U \in \ec{U V'}$.  By Lemma \ref{L:3RD_COLOR_PROPAGATES}, $\ec{U' V} = C_V \AND \ec{U V'} = C_U$.  This implies that $C_V = \ec{U' V'} = C_U$, a contradiction.

Case 2: $\tc{U} \neq \tc{V} \AND |\ec{UV}| = 2$.  Then we may assume there is $C \in \ec{UV}$ such that $C \nin \tc{U}$. Let $\ec{UV} = \{C,D\}$. Select $V_C \in \tr{V} \SUT C \in \ec{U V_C}$.  By Lemma \ref{L:3RD_COLOR_PROPAGATES}, $\ec{U V_C} = C$.  Now select $V_D \in \tr{V} \SUT D \in \ec{U V_D}$.  Observe that $C \nin \ec{U V_D}$ since that would imply that $D \nin \ec{U V_D} = C$.  Thus $\ec{U V_D} = D$.  Since $\ec{UV} = \{C,D\}$, we have accounted for every element of $\tr{V}$ and $\dom{V}{U}$, i.e. $(V,U) \in \SA_{n+1}$.

Case 3: $\tc{U} = \tc{V} = \{A,B\} \AND C \in \ec{UV} \setminus \{A,B\}$.  Fix $U_C \in \tr{U} \SUT C \in \ec{U_C V}$.  By Lemma \ref{L:3RD_COLOR_PROPAGATES}, $\ec{U_C V} = C$ and thus for every $V' \in \tr{V}, C \in \ec{UV'}$.  Applying Lemma \ref{L:3RD_COLOR_PROPAGATES} again, gives us that $\ec{UV'} = C$ and thus $\ec{UV} = C$. Thus $\{U,V\} \in \SE_{n+1}$.

Case 4: $\ec{UV} = \tc{U} = \tc{V}$.  If $\ndom{V}{U}$, apply $(3 \Rightarrow 1)$ from Lemma \ref{L:DOM_LEMMA} to get that $\dom{U}{V}$.
\eprf

\blem\label{L:PN_IS_DOM_TRANSITIVE}
Suppose $H$ has the tree property for $n$.  The weak components of $M_{n+1}(H)$ are transitive, and if $(U,V), (V, W) \in \SA_{n+1}$, then $(U,V) \sim_\Sigma (U,W)$.
\elem
\bprf
Let $U,V,W \in \SV_{n+1}$ such that $\dom{U}{V}$ and $\dom{V}{W}$.  By Lemma \ref{L:DOM_COLORS}, $|\tc{U}| = |\tc{V}| = 2$. We consider two cases: $\tc{U} \neq \tc{V}$ and $\tc{U} = \tc{V}$.

Suppose $\tc{U} \neq \tc{V}$ and let $A \in \tc{U} \setminus \tc{V}$.  Fix $U_A \in \tr{U}$ such that $\ec{U_A V} = A$ and $V_1, V_2 \in \tr{V}$ such that $\ec{V_1 W} \neq \ec{V_2 W}$.  Fix $W' \in \tr{W}$. We have that $U_A, V_1, W'$ forces $\ec{U_A W'} \subseteq \{A, \ec{V_1 W'}\}$ while $U_A, V_2, W'$ forces $\ec{U_A W'} \subseteq \{A, \ec{V_2 W'}\}$ and thus $\ec{U_A W'} = A$. Since $W'$ was arbitrary, $\ec{U_A W} = A$.  Note that since $|\ec{U_A W}| = 1$, we have ruled out the possibility that $\dom{W}{U}$.  By Lemma \ref{L:PN_IS_DOM_REGULAR}, we will be done if we can show that $|\ec{UW}| > 1$.  Observe  that we could also choose $U_B \in \tr{U}$ such that $\ec{U_B V} = B \neq A$.  If it happens that $B \nin \tc{V}$, by the same reasoning as above $\ec{U_B W} = B$ so that $\{A,B\} \subseteq \ec{UW}$ and thus $\dom{U}{W}$ and $\Sigma(U,V) = \Sigma(U,W)$.

Suppose then that $\tc{U} = \{A,B\}$ and $\tc{V} = \{B, C\}$.  We can now find $U_A, U_B \in \tr{U}$ such that $\ec{U_A W} = A$ and $\ec{U_B W} \subseteq \{B,C\}$.  Therefore $|\ec{UW}| > 1$ and by Lemma \ref{L:PN_IS_DOM_REGULAR} we have that $\dom{U}{W}$.  By Lemma \ref{L:DOM_COLORS}, $\ec{UW} = \tc{U} = \{A,B\}$ and thus $\ec{U_B W} = B$. Therefore, $\Sigma(U,V) = \Sigma(U,W)$.

We now consider the case $\tc{U} = \tc{V} = \{A,B\}$.  Note that $\ec{UW} \subseteq \{A,B\}$ since we may otherwise easily form a rainbow triangle.  We now have the setup for $(1 \Rightarrow 2)$ of Lemma \ref{L:DOM_LEMMA} with $\dom{U}{V}$ and have $x \in \SV \setminus (U \cup V)$ such that $\dom{U}{x}$, $\sig{U}{V} = \sig{U}{x}$, and $\ec{xV} = C \nin \{A,B\}$.  Applying part (1) of Lemma \ref{L:ABC} to $\dom{V}{W}$ and $x$, we have that $\ec{xW} = C$.  Now apply part (2) of Lemma \ref{L:ABC} with $\dom{U}{x}$ and $W$ to get that either $C \in \ec{UW}$ or $\dom{U}{W}$ and $\sig{U}{W} = \sig{U}{x} = \sig{U}{V}$.  We have already ruled out the former while the latter is what we sought to prove.
\eprf

\blem\label{L:PN_IS_DOM_NONINTERSECTING}
Suppose $H$ has the tree property for $n$. The weak components of $M_{n+1}(H)$ form rooted trees.
\elem
\bprf
After Lemma \ref{L:PN_IS_DOM_TRANSITIVE}, we need only show that for $U_1, U_2, V \in \SV_{n+1}$ distinct, if $\dom{U_1}{V}$ and $\dom{U_2}{V}$, then either $\dom{U_1}{U_2}$ or $\dom{U_2}{U_1}$. By Lemma \ref{L:PN_IS_DOM_REGULAR}, it suffices to show $|\ec{U_1 U_2}| > 1$.

Suppose $|\ec{U_1 U_2}| = 1$.  Observe that if $|\ec{U_1 V} \cup \ec{U_2 V} \cup \ec{U_1 U_2}| > 2$, then we must find a rainbow triangle in $U_1, U_2, V$.  Thus we may assume $\ec{U_1 U_2} = A \in \tc{U_1} = \tc{U_2} = \{A,B\}$. As in the proof of Lemma \ref{L:DOM_LEMMA}, since $\tc{U_1} = \tc{U_2}$, we may select $x \in \SV \setminus (U_1 \cup U_2)$ such that $\ec{U_1 x} = C \nin \{A,B\}$ and $A \in \ec{\brt{U_2} x}$ (or with the roles of $U_1$ and $U_2$ reversed).  We may then apply part (1) of Lemma \ref{L:ABC} with $\dom{U_1}{V}$ and $x$ to get that $\ec{Vx} = C$ and apply part (2) of Lemma \ref{L:ABC} with $\dom{U_2}{V}$ and $x$ to get that either $C \in \ec{U_2 x}$ or $\dom{U_2}{x}$.

First we consider the case $C \in \ec{U_2 x}$.  Let $k$ be minimal such that $C \in \ec{[\brt{U_2}]_k x}$.  If $k = 0$, by Lemma \ref{L:3RD_COLOR_WOUTSIDE}, $\ec{U_2 x} = C$, which contradicts our assumption that $A \in \ec{\brt{U_2} x}$.  Thus $k > 0$ and we may select $u \in [\brt{U_2}]_k \setminus [\brt{U_2}]_{k-1}$ such that $C \in \ec{ux}$.  We may apply part ($2^\prime$) of Lemma \ref{L:ABC} with $\dom{[\brt{U_2}]_{k-1}}{u}$ and $x$ to get that either $C \in \ec{[\brt{U_2}]_{k-1} x}$ or $\dom{[\brt{U_2}]_{k-1}}{x}$.  By the minimality of $k$, we must be in the latter case.  However, by assumptions that $\ec{U_1 U_2} = A$ and $\ec{U_1 x} = C$ and thus we may locate a rainbow triangle in $U_1$, $[\brt{U_2}]_{k-1}$, $x$.

We turn now to the second case, $\dom{U_2}{x}$.  We may apply part (2) of Lemma \ref{L:ABC} with $\dom{U_2}{x}$ and $U_1$ to get that either $C \in \ec{U_1 U_2}$ or $\dom{U_2}{U_1}$, which both contradict our assumption that $\ec{U_1 U_2} = A$. Thus $\dom{U_1}{U_2} \OR \dom{U_2}{U_1}$.
\eprf

Taking Lemmas \ref{L:PN_IS_DOM_REGULAR}, \ref{L:PN_IS_DOM_TRANSITIVE}, and \ref{L:PN_IS_DOM_NONINTERSECTING} together we have proved Theorem \ref{T:DECOMP}.

\section{Construction of Finite Gallai Multigraphs}\label{SEC:CONSTRUCTION}

In Section \ref{SEC:DECOMP}, we found that any reduced maximal Gallai multigraph $(G, \rho)$ can be decomposed into a sequence of mixed graphs, $(M_k(G), \rho', \sim_\Sigma),$ having certain properties. We now reverse this process to construct all finite Gallai multigraphs. An example of this construction is presented in Figure \ref{Example_Construct}.

\begin{construction}[Edge-colored Multigraph Construction $\Gamma$]\label{C:GAMMA}
Given a triple $(M = (\SV, \SE, \SA), \rho', \sim_\Sigma)$ with $M$ a complete mixed graph, $\rho'$ a list coloring of $\SE \cup \SA$ such that $|\rho'(e)| = 1$ for $e \in \SE$ and $|\rho'(e)| = 2$ for $e \in \SA$, and $\sim_\Sigma$ an equivalence on members of $\SA$ sharing an initial vertex, construct a complete, edge-colored multigraph $(G = (V, E), \rho)$ as follows:
\benum
\item Replace each $u \in \SV$ with $(G_u = (E_u, V_u), \rho_u)$, a uniformly colored Gallai multigraph such that if $(u,v) \in \SA$ for some $v \in \SV$, then $|V_u| \geq 2$ and $\rho_u[V_u V_u] = \rho'((u,v))$.
\item For each $u \in \SV$, $\rho \restrict_{E_u} := \rho_u$.
\item For $u, v \in \SV$ distinct, connect $V_u \AND V_v$ as follows:
\benum
\item if $\{u,v\} \in \SE$, then $\rho(\{w_1, w_2\}) = \rho'(\{u, v\})$ for $w_1 \in V_u \AND w_2 \in V_v$;
\item if $(u,v) \in \SA$, then define $\rho$ such that
\benum
	\item $\dom{V_u}{V_v}$,
	\item $\rho[V_u V_v] = \rho'((u,v))$, and
	\item $\Sigma(V_u, V_v) = \Sigma(V_u, V_w)$ whenever $(u,v) \sim_{\Sigma} (u,w)$.
\eenum
It is easy to see that $\rho$ can be defined in this way whenever $|V_u| \geq 2$ as required above.
\eenum
\eenum
\end{construction}

We write $\Gamma((M, \rho', \sim_\Sigma))$ for the family of edge-colored multigraphs resulting from all possible choices of $\{(G_u, \rho_u)\}$ in step (1) and permissible definitions of $\rho$ in step (3.b).

We use the following notation in the construction of Gallai multigraphs.
\bitem
\item $\SG = \{(G_i, \rho_i)\}$ is the family of all finite Gallai multigraphs.
\item $\SG_r = \{(G_i, \rho_i)\}$ is the family of all finite reduced Gallai multigraphs.
\item $\SG_r ^+$ is the family of all induced subgraphs of maximal members of $\SG_r$.
\item $\SM = \{ (M_n(H), \rho, \sim_\Sigma) : H \in \SG_r^+ \AND n \in \N \}$.
\item $\ST$ is the family of rooted trees in $\SM$.
\item $\SM^\ast$ is the family of triples $(M, \rho, \sim_\Sigma)$ such that $\Gamma((M, \rho, \sim_\Sigma)) \subseteq \SG$.
\item $\ST^\ast$ is the family of rooted trees in $\SM^\ast$.
\eitem
Lastly, we let $\SG(n) = \{(G = (V,E), \rho) : |V| \leq n\}$ and likewise for each of the families defined above. We will construct a family $\SM^\prime$ and show that
\bitem
\item $\SM \subseteq \SM^\prime \subseteq \SM^\ast$ and
\item $\SG_r ^+ \subseteq \Gamma[\SM] \subseteq \Gamma[\SM^\prime] \subseteq \Gamma[\SM^\ast] \subseteq \SG$.
\eitem
In particular, we will construct a family of Gallai multigraphs,$\Gamma[\SM^\prime]$, containing the reduced maximal ones.

\showpic{4.5}{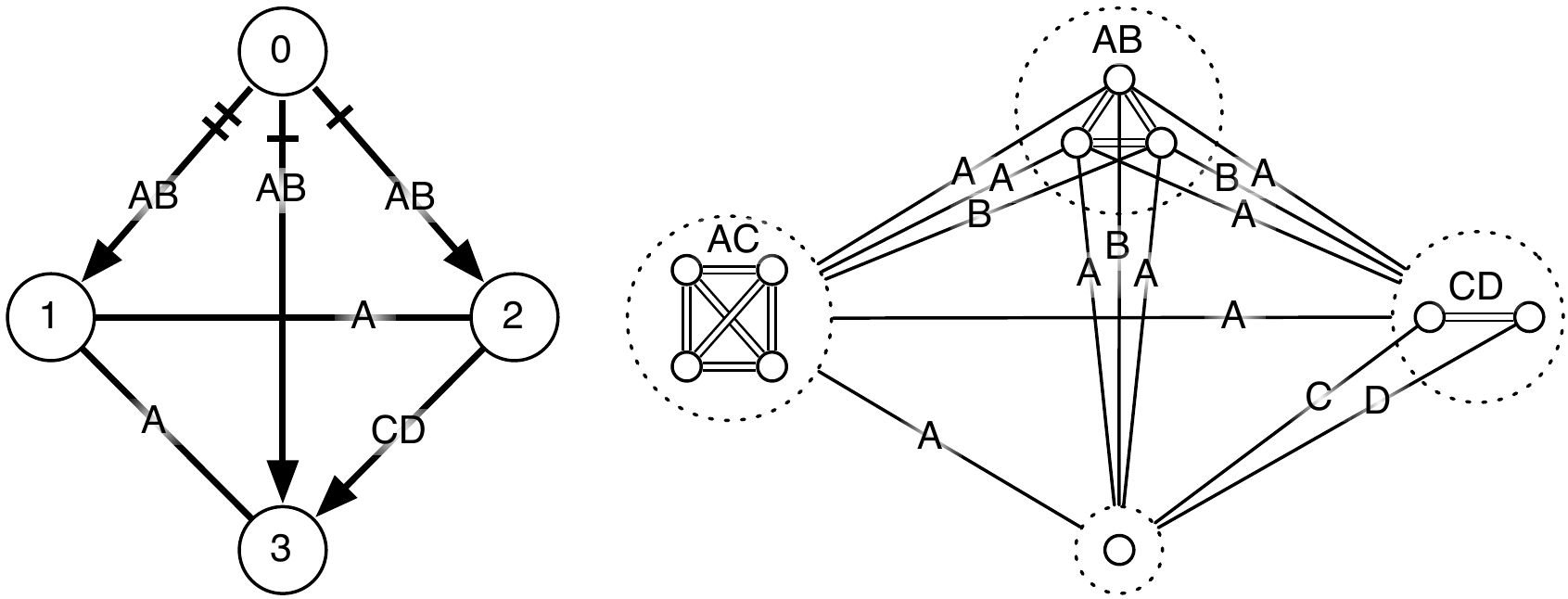}{$(M, \rho, \sim_\Sigma)$ (left) and a member of $\Gamma((M, \rho, \sim_\Sigma))$ (right)}{Example_Construct}

\begin{lemma}\label{L:TRIANGLE_CASES}
Suppose $(M_n (G) = (\SV, \SA, \SE), \rho, \sim_\Sigma) \in \SM$. If $\rho'$ is a coloring of $\SA \cup \SE$ such that $\rho'(e) \in \rho(e)$ for all $e \in \SA \cup \SE$ and $\rho'(e_1) = \rho'(e_2)$ whenever $e_1 \sim_\Sigma e_2$, then $(M_n(G), \rho')$ lacks rainbow triangles.
\end{lemma}
\begin{proof}
Fix $U, V, W \in \SV$ distinct. Recall that $U,V, \AND W$ are disjoint subsets of vertices of $G$ and that $(U,V) \in \SA$ if and only if $\dom{U}{V}$ in $G$. We now consider each of the general cases presented in Figure \ref{TRIANGLE_CASES}. Fix $v \in V \AND w \in W$.

\showpic{4.7}{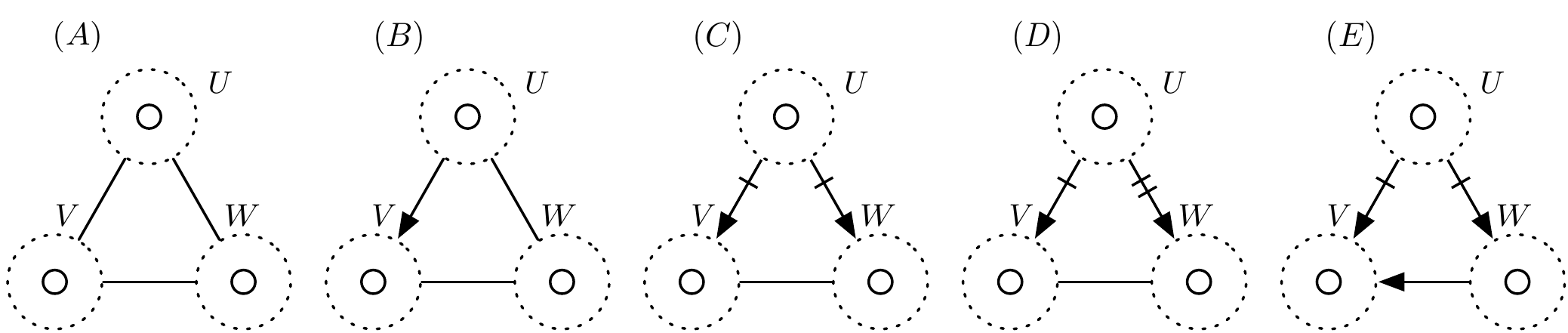}{General possible relations between $U,V, \AND W$.}{TRIANGLE_CASES}

Case (A): Select $u \in U$ and note that $\rho'$ colors the triangle $U, V, W$ just as $u,v,w$ is colored in $G$.

Case (B): Select $u \in U$ such that $\ec{uV} = \rho'((U,V))$. Again, $U,V,W$ is colored in the same way as $u,v,w$.

Cases (C) and (E): Since $(U,V) \sim_\Sigma (U,W)$, $\rho'((U,V)) = \rho'((U,W))$.

Case (D): Since $(U,V) \not\sim_\Sigma (U,W)$, we know $\Sigma(U,V) \neq \Sigma(U,W)$ and we may thus select $u \in U$ such that $\ec{uV} \neq \ec{uW}$. It might happen that $\ec{uV} = \rho'((U,V))$ and $\ec{uW} = \rho'((U,W))$. In this case, we again note that $\rho'$ colors $U,V,W$ in the same way that $u,v,w$ is colored in $G$. Suppose then that we are in the alternate case: $\ec{uV} = \rho'((U,W))$ and $\ec{uW} = \rho'((U,V))$. This however also forces $\ec{VW} = \rho'(\{V,W\}) \in \{\rho'((U,V)), \rho'((U,W))\}$.
\end{proof}

\begin{lemma}\label{GAMMA_WORKS}
$\SG_r ^+ \subseteq \Gamma[\SM] \subseteq \SG$ and thus $\SM \subseteq \SM^\ast$.
\end{lemma}
\begin{proof}
To see that $\Gamma[\SM] \subseteq \SG$, let $(M = (\SV, \SA, \SE), \rho, \sim_\Sigma) \in \SM$ and fix $(G = (V,E), \rho') \in \Gamma((M, \rho, \sim_\Sigma))$. Let $\{(G_u =(E_u, V_u), \rho_u)\}$  be the family of Gallai multigraphs used in step (1) of the construction of $(G, \rho')$. We now show that $(G, \rho')$ lacks rainbow triangles.

First note that for a given triangle, $u,v,w$, if any two of the vertices correspond to the same vertex of $M$, i.e. fall in the same $V_i$, then $u,v,w$ must have a repeated color.

Therefore we only need to consider the triangles formed by vertices from different $G_i$. In particular, we may form $V' \subseteq V$ by selecting a single vertex from each of $V_i$ and consider $(V', E', \rho')$, the induced graph of $G$ by $V'$.

We will be done if we show that $(V',E', \rho')$ lacks rainbow triangles. Notice that this triple is equivalent to an edge-coloring of $M_n(G)$ of the form described in Lemma \ref{L:TRIANGLE_CASES} and thus $(V',E', \rho')$ lacks rainbow triangles.

Finally, note that $H \in \Gamma((M_1(G), \rho_1, \sim_\Sigma ))$ for each $H \in \SG_r ^+$ and thus $\SG_r ^+ \subseteq \Gamma[\SM]$.

\end{proof}

We now construct a family $\SM^\prime$ such that $\SM \subseteq \SM^\prime \subseteq \SM^\ast$.

\begin{construction}[Forest Construction $\Delta_F$]\label{C:DELTA_F}
Suppose we have at our disposal $\SN \subseteq \SM^\ast$. Set $\ST_\SN = \SN \cap \ST$ and form a triple $(M' = (\SV', \SE', \SA'), \rho', \sim_{\Sigma'})$ as follows:

\benum
\item Fix $(M = (\SV, \SE, \SA), \rho, \sim_\Sigma) \in \SN$.
\item For each $u \in \SV$, fix $(T_u = (V_u, E_u, A_u), \rho_u, \sim_{\Sigma_u}) \in \ST_\SN$ such that if $(u,v) \in \SA$ for some $v \in \SV$, then
\benum
\item $\tc{T_u} = \rho((u,v))$ (this is just an issue of labeling) and
\item $|V_u| \geq 2$. 
\eenum
\item Set $\SV' := \cup_{u \in \SV} V_u$.
\item Set $\SA' := \cup_{u \in \SV} A_u$.
\item Set $\SE' := (\cup_{u \in \SV} E_u) \bigcup \left\{ \{ w_1, w_2 \} : w_1 \in V_u, w_2 \in V_v, \AND u \neq v\right\}$.
\item Define $\rho'$ as follows:
\benum
\item $\rho'\restrict_{A_u \cup E_u} := \rho_u$ for all $u \in \SV$;
\item for $\{w_1, w_2\} \in \SE'$ where $w_1 \in V_u, w_2 \in V_v, \AND u \neq v$,
\benum
\item if $\{u,v\} \in \SE$, then $\rho'(\{w_1, w_2\}) := \rho(\{u,v\})$;
\item if $(u,v) \in \SA$, then define $\rho'$ on $V_u V_v$ so that
\benum
\item $\dom{V_u}{V_v}$,
\item $\rho'[V_u V_v] = \rho((u,v))$, and 
\item $\Sigma(V_u, V_v) = \Sigma(V_u, V_w)$ whenever $(u,v) \sim_\Sigma (u,w)$

\eenum
\eenum
\eenum
\eenum
\end{construction}

Key to this construction is understanding when and how step (6.b.ii) can be accomplished. We call this the signature configuration problem and resolve it in Section \ref{ALL_SIGNATURES}

We write $\Delta_F (\SN)$ for the family of all mixed graphs resulting from one iteration of this construction beginning with $\SN$. Notice that $\Delta_F (\SN)$ constructs no rooted trees that were not already present in $\SN$, and we will therefore need a separate construction for trees. We present this tree construction in Section \ref{MAKING_TREES} and write $\Delta_T(\SN)$ for the resulting family of trees. Finally, we set $\Delta(\SN) := \Delta_F(\SN) \cup \Delta_T(\SN)$.

We let $\SM_0$ be the family of elements in $\SM$ having no directed edges (and are thus essentially Gallai graphs) and for $n \geq 0$, set $\SM_{n+1} := \Delta(\SM_n)$ and $\SM^\prime := \cup_{n=0}^\infty \SM_n$.

\begin{theorem}\label{T:CONSTRUCTION}
$\SM \subseteq \SM^\prime \subseteq \SM^\ast$ and thus $\SG_r ^+ \subseteq \Gamma(\SM^\prime) \subseteq \SG$.
\end{theorem}

We prove Theorem \ref{T:CONSTRUCTION} after addressing the signature configuration problem in Section \ref{ALL_SIGNATURES} and describing the tree construction $\Delta_T$ in Section \ref{MAKING_TREES}.

\subsection{Signature Configuration Problem}\label{ALL_SIGNATURES}

Here we address the following basic problem arising in step (6.b.ii) of the $\Delta_F$ construction: given $T \in \ST$ and a set of vertices $U$, describe all possible ways to join $T \AND U$ with two colors so that $\dom{T}{U}$ and $T \cup U$ lacks rainbow triangles.

First observe that, since $\dom{T}{U}$, the set of vertices $U$ is irrelevant and that we may just as well assume there is only a single vertex $u$. Now suppose we would like to define dominance between $T = (V,E,A)$ and $u$ using the colors $A \AND B$. Up to relabeling, we may assume $\tc{T} = \{A,B\}$. The signature of $\dom{T}{u}$ is determined by the map $\Sigma:V \rightarrow \{A,B\}$ given by $v \mapsto \ec{vu}$.

For $v_1, v_2 \in V$, if $\ec{v_1 v_2} \not\subseteq \{A, B\}$, then to avoid a rainbow triangle it must be the case that $\Sigma(v_1) = \Sigma(v_2)$, i.e. $\ec{v_1 u} = \ec{v_2 u}$. We therefore partition $V$ by first defining the relation $v_1 \sim v_2$ if $\ec{v_1 v_2} \not\subseteq \{A, B\}$ and then extending $\sim$ to an equivalence.

\begin{lemma}\label{L:DEFINING_DOMINANCE}
$T \cup u$ lacks rainbow triangles if and only if $\Sigma$ is constant on $\sim$ classes.
\end{lemma}
\begin{proof}
The statement follows directly from the definition of $\sim$.
\end{proof}

\begin{theorem}
For $T = (V,E,A) \in \ST$ and $M \in \SM$, the ways of joining $T$ and $M$ so that $\dom{T}{M}$ and $T \cup M$ lacks rainbow triangles correspond exactly to choice functions $\Sigma: (V/\sim) \twoheadrightarrow \tc{T}$.
\end{theorem}
\begin{proof}
The only difference between this statement and Lemma \ref{L:DEFINING_DOMINANCE} is that $T$ is now dominating a set of vertices rather than a single vertex, but since $\dom{T}{M}$, whatever colors may be present in $M$ cannot form rainbow triangles with vertices from $T$. The requirement that $\Sigma$ be onto corresponds to the fact that dominance requires multiple colors and $|\tc{T}| = 2$.
\end{proof}

\subsection{Tree Construction $\Delta_T$}\label{MAKING_TREES}

We now describe how to construct rooted trees of size $n+1$ out of a mixed graph of size $n$. Here we use the fact that Theorem \ref{T:DECOMP} holds for all of $\SG_r ^+$, rather than just the maximal members.

We would like to construct $T = (V,E,A) \in \ST$ using elements of $\SM$ having fewer than $|V|$ vertices. Perhaps the most natural approach would be to consider the collection of subtrees $\{T_i\}$ in $V \setminus \rt{T}$ as in the left-most image of Figure \ref{BUILDING_TREES} and then list all possible ways to assign signatures to $\dom{\rt{T}}{T_i}$. One quickly realizes that with this approach you must also describe how each pair of $\{T_i\}$ is related.

\showpic{4.5}{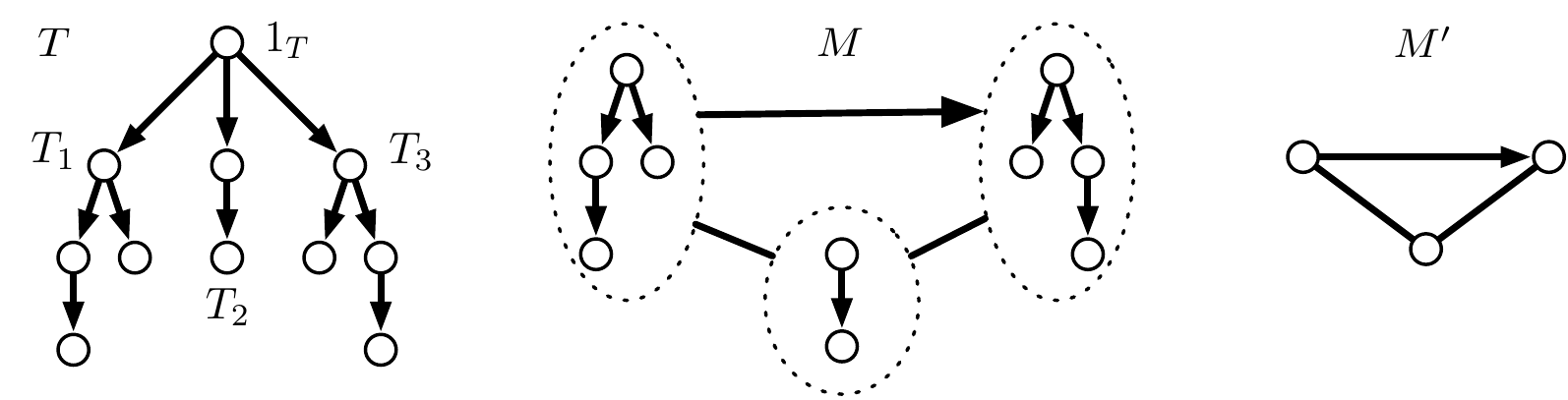}{Constructing a tree from a smaller forest.}{BUILDING_TREES}

Notice, however, that since $T \in \ST \subseteq \SM$, $V$ corresponds to the induced subgraph of some reduced maximal Gallai multigraph, and therefore $V \setminus \rt{V}$ does as well. We may thus use Theorem \ref{T:DECOMP} to show that $V \setminus \rt{V}$ decomposes into a mixed graph in $\SM$. In particular, Lemma \ref{L:PRUNED_TREES} shows that if we remove the root from $T$, we are left with $M \in \SM$ whose weak components are $\{T_i\}$, as in the center image of Figure \ref{BUILDING_TREES}. We may now collapse $M$ into $M'$, as in the right-most image of Figure \ref{BUILDING_TREES}.

\begin{lemma}\label{L:PRUNED_TREES}
Let $(G = (V,E), \rho) \in \SG_r ^+$ and $M_k(G) = (\SV_k, \SE_k, \SA_k)$ for all $k \leq n$. Let $V_T \in \SV_n$ and let $H$ be the induced subgraph of $G$ by $V \setminus V_T$. Let $M_k(H) = (\SV_k ', \SE_k ', \SA_k ')$ for all $k \leq n$. Set $V_T(k) = \{W \in \SV_k : W \subseteq V_T\}$. Then $\SV_k = \SV' _k \cup V_T (k)$ for all $k \leq n$.
\end{lemma}

\begin{proof}
The claim holds for $k=0$ since $\SV_0 = V$, $\SV'_0 = V \setminus V_T$, and $V_T(0) = V_T$. Let $k + 1 \leq n$ be minimal such that $\SV_{k+1} \neq \SV'_{k+1} \cup V_T(k+1)$. By definition, $V_T(k+1)$ agrees with $\SV_{k+1}$ on 
the $V_T$ portion of $V$. It must be the case then that $\SV_{k+1} \setminus V_T(k+1) \neq \SV'_{k+1}$. That is, $\SV_{k+1} \setminus V_T(k+1)$ and $\SV'_{k+1}$ represent two different partitions of $V \setminus V_T$. However, by the minimality of $k+1$, $\SV_{k} \setminus V_T(k) = \SV'_{k}$. Note that $\SV'_{k+1}$ is completely determined by the dominance relations between members of $\SV'_k$. Likewise, since $V_T \in \SV_n$, none of the members of $\SV'_k$ is in a dominance relation with any member of $V_T(k)$ and thus $\SV'_{k+1} \setminus V_T(k+1)$ is also determined by the dominance relations between members of $\SV'(k)$. Thus $\SV_{k+1} \setminus V_T(k+1) = \SV'_{k+1}$, a contradiction.
\end{proof}

We will be done once we specify all ways in which we can assign signatures to the directed edges from $\rt{T}$ to the vertices of $M'$. In particular, a choice of signatures corresponds exactly to a partition of the vertices of $M'$. What then are the necessary and sufficient conditions on this partition to produce a valid choice of signatures? Lemma \ref{L:TREE_MAKING} answers this question in much the same way as Lemma \ref{L:DEFINING_DOMINANCE} did for the signature configuration problem.

\begin{lemma}\label{L:TREE_MAKING}
For $(T = (V,E,A), \rho, \sim) \in \ST$ and $\{u,v\} \in E$, if $\rho(\{u, v\}) \not\subseteq \tc{T}$, then $\Sigma(\rt{T}, u) = \Sigma(\rt{T}, v)$.
\end{lemma}
\begin{proof}
Suppose $\tc{T} = \{A,B\}$ and $C \in \rho(\{u,v\}) \setminus \{A,B\}$.  If $\Sigma(\rt{T}, u) \neq \Sigma(\rt{T}, v)$, then we may form a rainbow triangle by selecting $A$ for $(\rt{T}, u)$, $B$ for $(\rt{T}, v)$, and $C$ for $\{u,v\}$.
\end{proof}

\begin{construction}[Tree Construction: $\Delta_T$]\label{C:TREES}
For $\SN \subseteq \SM^\ast$, fix $(M = (V,E,A), \rho, \sim_\Sigma) \in \SN$ and colors $\{A,B\}$. Let $V = \cup V_i$ be the vertex partition of $M$ into weak components. Let $\rt{T}$ be a new vertex and form $(T := (V \cup \rt{T}, E', A'), \rho', \sim_{\Sigma^\prime})$ as follows:

\benum
\item $E' := E$;
\item $A' := A \cup (\{\rt{T}\} \times \{V\})$;
\item $\rho'(e) := \{A,B\}$ for all $e \in \{\rt{T}\} \times \{V\}$ and $\rho' \restrict E \cup A := \rho$;
\item Define $\sim_{\Sigma^\prime}$ on $A'$ by
\benum
	\item $(\rt{T}, u) \sim_{\Sigma^\prime} (\rt{T}, v)$ whenever $(u,v) \in A$,
	\item if $\ec{V_i V_j} \not\subseteq \{A,B\}$, then $(\rt{T}, u) \sim_{\Sigma^\prime} (\rt{T}, v)$ for all $u \in V_i, v \in V_j$,
	\item if $(u,v_1) \sim_\Sigma (u,v_2)$, then $(u,v_1) \sim_{\Sigma^\prime} (u,v_2)$ for all $u,v_i \in V$.
\eenum
\eenum

\end{construction}

We write $\Delta_T(\SN)$ for the family of mixed graphs resulting from all possible choices of $(M, \rho, \sim_\Sigma) \in \SN$ and possible definitions of $\sim_\Sigma$ in step (4).

\begin{theorem}\label{T:TREE_CONSTRUCTION_WORKS}
$\ST(n+1) \subseteq \Delta_T(\SM(n)) \subseteq \Delta_T(\SM^\ast) \subseteq \ST^\ast$.
\end{theorem}

\begin{proof}
Since $\SM(n) \subseteq \SM^\ast$, we have $\Delta_T(\SM(n)) \subseteq \Delta_T(\SM^\ast)$.  We now show that $\Delta_T(\SM^\ast) \subseteq \ST^\ast$. Fix $M \in \SM^\ast$, $T \in \Delta_T(M)$, and $G \in \Gamma(T)$.

By construction, it is clear that $T$ is a tree with root $\rt{T}$ so we only need to verify $T \in \ST^\ast$, i.e. $G \in \SG$. Since $M \in \SM^\ast$, the portion of $G$ corresponding to $M$, that is $G$ with $\rt{T}$ removed, will certainly lack rainbow triangles. Likewise, if a triangle falls entirely in the portion of $G$ corresponding to $\rt{T}$, then it will have at most two colors. We thus only need to consider the general types of triangles presented in Figure \ref{TREE_CASES}.
\showpic{4.5}{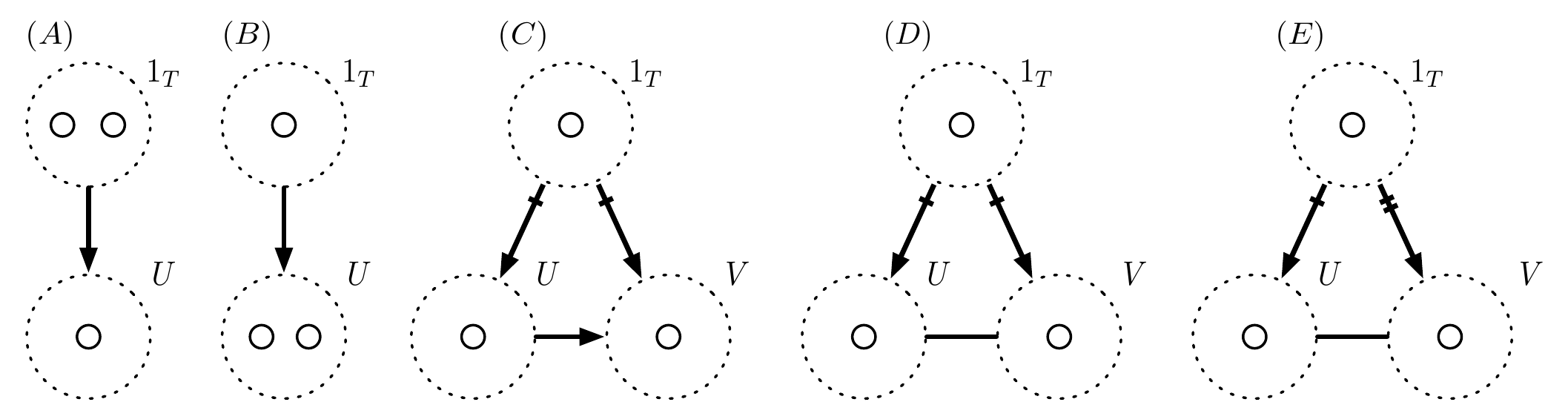}{}{TREE_CASES}

Case (A) cannot lead to a rainbow triangle since $\ec{\rt{T} \rt{T}} \subseteq \tc{T} = \ec{\rt{T} U}$. Likewise, in case (B), since $\dom{\rt{T}}{U}$, the vertex in $\rt{T}$ must be connected to $U$ by a single color. Cases (C) and (D) cannot contain a rainbow triangle since $\Sigma(\rt{T}, U) = \Sigma(\rt{T}, V)$.

Finally, by construction, if $\rho(U,V) \not\subseteq \tc{T}$, then $\Sigma(\rt{T},U) = \Sigma(\rt{T},V)$. Since in case (E), $\Sigma(\rt{T},U) \neq \Sigma(\rt{T},V)$ it must be the case that $\rho(U,V) \subseteq \tc{T}$ and thus the figure cannot contain a rainbow triangle. We have thus shown that $\Delta_T(\SM^\ast) \subseteq \ST^\ast$.

We now show $\ST(n+1) \subseteq \Delta_T(\SM(n))$. Fix $T \in \ST(n+1)$. Since $\ST(n+1) \subseteq \SM$, $T = M_k(H)$ for some $H \in \SG_r ^+$. Let $H'$ be the subgraph of $H$ induced by removing the vertices corresponding to $\rt{T}$. Again, $H' \in \SG_r ^+$ and thus $M_k(H') \in \SM$. By Lemma \ref{L:PRUNED_TREES} however, $M_k(H')$ will have one less vertex than $T$ and thus $M_k(H') \in \SM(n)$. Since $T \in \Delta_T(M_k(H')) \subseteq \Delta_T(\SM(n))$, we have $\ST(n+1) \subseteq \Delta_T(\SM(n))$.
\end{proof}

\subsection{Proof of Theorem \ref{T:CONSTRUCTION}}

We may now show that $\SM \subseteq \SM^\prime \subseteq \SM^\ast$.

\begin{proof}
It is clear that $\SM \subseteq \SM^\prime$.

Certainly $\SM_0 \subseteq \SM^\ast$ since $\SM_0$ is just the family of Gallai graphs. Now consider $\SM_{n+1} = \Delta_T(\SM_n) \cup \Delta_F(\SM_n)$ where $\SM_n \subseteq \SM^\ast$. Theorem \ref{T:TREE_CONSTRUCTION_WORKS} shows that $\Delta_T(\SM_n) \subseteq \SM^\ast$.

Fix $(M = (V,E,A), \rho, \sim_\Sigma) \in \Delta_F(\SM_n)$ and $(G, \rho) \in \Gamma(M)$. Recall that $M$ was constructed by replacing vertices of $M' \in \SM_{n}$ with trees $\{T_i\} \subseteq \SM_n$ and joining these trees subject to certain constraints. Recall further than $G$ was formed by replacing the vertices of $M$, which as just mentioned can be thought of as the vertices of the collection $\{T_i\}$, with complete edge-colored multigraphs, say $\{G_i\} \in \SG$.

Now consider a triangle $u,v,w$ in $G$. Let $u$ fall in the portion of $G$ corresponding to $T_u$, $v$ in $T_v$, and $w$ in $T_w$. Up to relabeling, we may then assume we are in one of the following cases:

\benum
\item $T_u = T_v = T_w$,
\item $T_u, T_v, \AND T_w$ are distinct,
\item $T_v = T_w$ and $\dom{T_u}{T_v}$, or
\item $T_u = T_v$ and $\dom{T_u}{T_w}$.
\eenum

Cases (1) and (2) are resolved by the facts that $T_u \in \SM^\ast$ and $M' \in \SM^\ast$, respectively. Case (3) is straightforward since $\Sigma(T_u, T_v) = \Sigma(T_u, v) = \Sigma(T_u, w)$ and thus $\ec{uv} = \ec{uw}$.

Case (4) requires some consideration of $\Sigma(T_u, T_w)$. If it happens that $\Sigma(T_u, T_w)(u) = \Sigma(T_u, T_w)(v)$, then $\ec{uw} = \ec{vw}$ and we are done. Suppose then that $\Sigma(T_u, T_w)(u) \neq \Sigma(T_u, T_w)(v)$. By construction, this implies that $u$ and $v$ are not connected by a path in $T_u$ whose colors fall outside $\tc{T_u}$. In particular, $\ec{uv} \subseteq \tc{T_u}$. Since we also know that $\ec{uw}, \ec{vw} \subseteq \tc{T_u}$ and $|\tc{T_u}| = 2$, we are done.

\end{proof}

\section{Open Problems}\label{OPEN_PROBLEMS}

We approached the topic of Gallai multigraphs in an attempt to generalize the construction of Gallai graphs for graphs lacking rainbow $4$-cycles. We remain interested in this problem and in the following somewhat more general questions suggested by the multigraph perspective:

\benum
\item Can $\Delta$ be generalized to construct finite multigraphs lacking rainbow $n$-cycles for a fixed $n$?
\item Is there a construction of all finite (not necessarily complete) graphs or multigraphs lacking rainbow triangles?
\eenum

\section{Acknowledgments}
I would like to thank Rick Ball and Petr Vojt\u{e}chovsk\'y for suggesting to me the topic of Gallai graphs, and Petr for his many helpful comments. I would also like to thank Peter Keevash for informing me of \cite{Mubayi}.

\end{document}